\documentclass{amsart}
\usepackage{amssymb}
\usepackage{amsmath}
\usepackage{amscd}
\usepackage{amsbsy}
\usepackage{comment}
\usepackage[matrix,arrow]{xy}
\usepackage[backref]{hyperref}
\usepackage{amsrefs}
\usepackage{tikz-cd}
\usepackage{listings}
\usepackage{xcolor}

\DeclareMathOperator{\PGL}{PGL}

\DeclareMathOperator{\Aut}{Aut}

\DeclareMathOperator{\GL}{GL}

\DeclareMathOperator{\Gal}{Gal}
\DeclareMathOperator{\norm}{Norm}

\newcommand{\Q}{{\mathbb Q}}

\newcommand{\F}{{\mathbb F}}

\newcommand{\cO}{\mathcal{O}}

\newcommand{\fq}{\mathfrak{q}}

\newcommand{\rhobar}{\overline{\rho}}

\newtheorem{theorem}{Theorem}[section]

\newtheorem{remark}[theorem]{Remark}

\newtheorem{proposition}[theorem]{Proposition}
\theoremstyle{definition}
\newtheorem*{assumption}{Assumptions}

\title[Irreducibility of Galois representations]{Irreducibility of mod $p$ Galois representations of elliptic curves with multiplicative reduction over number fields}
\begin{document}
	
	\author{Filip Najman}
	
	\address{Department of Mathematics, Faculty of Science, University of Zagreb\\
		Bijeni\v cka cesta 30\\
		10000 Zagreb\\
		Croatia}
	\email{fnajman@math.hr }
	\author{George C. \c Turca\c s}
	\address{``Simion Stoilow" Institute of Mathematics of the Romanian Academy\\
		014700 \\ Bucharest\\
		Romania}
	\address{Babe\c s-Bolyai University,
		Faculty of Mathematics and Computer Sciences\\
		1 Kog\u alniceanu Street\\
		400084, Cluj-Napoca\\
		Romania}
	\email{george.turcas@math.ubbcluj.ro}
	
	\date{\today}
	\thanks{Najman is supported by the QuantiXLie Centre of Excellence, a project
		co-financed by the Croatian Government and European Union through the
		European Regional Development Fund - the Competitiveness and Cohesion
		Operational Programme (Grant KK.01.1.1.01.0004) and by the Croatian Science Foundation under the project no. IP-2018-01-1313.  The second author is supported by a Bitdefender Postdoctoral scholarship at the ``Simion Stoilow" Institute of Mathematics of the Romanian Academy}
	
	
	\maketitle

	\begin{abstract}  In this note we prove that for every integer $d \geq 1$, there exists an explicit constant $B_d$ such that the following holds. Let $K$ be a number field of degree $d$, let $q > \max\{d-1,5\}$ be any rational prime that is totally inert in $K$ and $E$ any elliptic curve defined over $K$ such that $E$ has potentially multiplicative reduction at the prime $\fq$ above $q$. Then for every rational prime $p> B_d$, $E$ has an irreducible mod $p$ Galois representation. This result has Diophantine applications within the ``modular method". We present one such application in the form of an Asymptotic version of Fermat's Last Theorem that has not been covered in the existing literature.
		
	\end{abstract}
	
	\section{Introduction}

	Throughout this article $K$ will denote a number field, $G_K = \Gal(\overline K/K)$ its absolute Galois group and $E$ an elliptic curve defined over $K$. For a rational prime $p$, we are going to write $\rhobar_{E,p}$ for the representation
	$$ \rhobar_{E,p}: G_K \to \Aut(E[p]) \cong \GL_2(\F_p)$$
	arising from the action of $G_K$ on the $p$-torsion points in $E(\overline{K})$. In the language of Galois representations, $E$ having a $p$-isogeny defined over $K$ is equivalent to $\rhobar_{E,p}$ being reducible.
	
	When $K= \Q$, it follows from Mazur's Theorem \cite[Theorem 1]{Mazur1978} that for $p > 163$ the representation  $\rhobar_{E,p}$ is irreducible for all elliptic curves $E$. For a general number field $K$, the question of whether  there is a constant $B_K$ such that $\rhobar_{E,p}$ is irreducible for all primes $p>B_K$ and all elliptic curves $E$ without complex multiplication is an active topic of research. The existence of such a constant $B_K$ has not been proved for any number field other than $\Q$.
	
	The ``modular approach" is a powerful method for showing that certain Diophantine equations do not have solutions using Galois representations of elliptic curves \cite{Siksek2007}. Absolute irreducibility of the mod $p$ Galois representations associated to Frey elliptic curves is a necessary hypothesis for successful applications of this method. With these type of Diophantine applications in mind, Freitas and Siksek \cite{FreSikCriteria} proved the following theorem.
	
	\begin{theorem}[\cite{FreSikCriteria}*{Theorem 2}] \label{trealired} Let $K$ be a totally real Galois field. There is an explicit effective constant $B_K$, depending only on $K$ such that for any rational prime $p> B_K$ and elliptic curve $E$ over $K$ semistable at all $\mathfrak p \mid p$, the representation $\rhobar_{E,p}$ is irreducible.
	\end{theorem}
	
	\noindent\textbf{Remark.}
	When $K$ is totally real and $p$ is odd, the existence of a complex conjugation in $G_K$ whose image under $\rhobar_{E,p}$ is similar to $\left( \begin{array}{cc} 1 & 0 \\ 0 & -1 \end{array} \right)$ implies that if $\rhobar_{E,p}$ is irreducible, then it is absolutely irreducible. This is not true if $E$ is defined over a general number field $K$.
	
	Recent progress in (potential) modularity over number fields that are not totally real sparked an interest into attacking Diophantine equations over number fields via the modular method. In \cite{FrSiKr,ASiksek} the authors discuss the Fermat equation with prime exponent over general number fields $K$
	\begin{equation} \label{fermateq}
	a^p+b^p+c^p=0,
	\end{equation}
	where $a,b,c \in K$ and $p$ is a rational prime, assuming two conjectures from the Langlands programme. The results in the aforementioned papers are in the direction of the ``Asymptotic Fermat's Last Theorem". To be precise, under various hypotheses on $K$, the authors of these papers prove that for $p$ larger than a constant $C_K$, the equation \eqref{fermateq} does not have non-trivial solutions. Under the same two conjectures, Kara and \"{O}zman \cite{KarOz} proved asymptotic versions of the so called ``Generalized Fermat equation" over number fields.
	
	In \cite{turcas2018} and \cite{turcas2020}, the second author considered the Fermat equation over quadratic imaginary fields of class number one, assuming \cite{turcas2020}*{Conjecture 2.2}. The latter is commonly called ``Serre's modularity conjecture" and is one of the two conjectures assumed in \cite{ASiksek,FrSiKr}.
	
	The hypotheses required for applying \cite{turcas2020}*{Conjecture 2.2} include the absolute irreducibility of the mod $p$ Galois representations $\rhobar_{E,p}$. Proving results such as Theorem \ref{trealired} for general number fields $K$ is a hopeless task. This is due to the possible existence of elliptic curves defined over $K$ whose CM-fields are contained in $K$. The representations $\rhobar_{E,p}$ of such curves are reducible for infinitely many (half of the) values of $p$ and irreducible but absolutely reducible for almost all the remaining ones.


Our main results in this paper are as follows.

	\begin{theorem} \label{mainquad}
		Let $K$ be a quadratic field and let $q>5$ be a rational prime that is inert in $K$. Suppose an elliptic curve $E/K$ has multiplicative reduction at the prime $\fq$ of $K$ above $q$ and let $p>71$ be a prime.  Then $E$ has an irreducible mod $p$ representation.
	\end{theorem}

	For number fields of degree which is larger than two, we obtained the subsequent theorem.
	
	\begin{theorem} \label{maind}
		Let $K$ be a number field of degree $d>2$ and let $q> d-1$ be a rational prime that is totally inert in $K$. Suppose an elliptic curve $E/K$ has multiplicative reduction at the prime $\fq$ of $K$ above $q$ and let $p>65 (2d)^6$ be a prime. Then $E$ has an irreducible mod $p$ representation.
	\end{theorem}
	

The results above imply that there exists a constant $C_d$, depending only on the degree $d=[K:\Q]$ such that for all elliptic curves with multiplicative reduction at a totally inert prime $q>d-1$ it will follow that $\rhobar_{E,p}$ is irreducible for all $p$ larger than $C_d$. 


	
	We now explain that, when carrying out the modular approach to Diophantine equations, results such as the ones above are very valuable. A Frey elliptic curve $E$ is semistable outside a fixed finite set of primes. That is, $E$ has good or multiplicative reduction at prime ideals outside this set.  For $p > C_K$ and $\rhobar_{E,p}$ (absolutely) irreducible one can proceed with the modular method as usual. In the case that $\rhobar_{E,p}$ is reducible, together with some additional hypothesis, the theorems above can help proving that the primes of multiplicative reduction for $E$ belong to a finite fixed set. The $j$-invariant of $E$ is hence integral outside this set. Moreover, there are explicit formulas for $j(E)$ depending on the solution to the Diophantine equation in question. Knowing that the denominators of $j(E)$ are supported on a restricted set of primes can therefore lead to a complete resolution of the equation.
	
	To emphasize this phenomenon, recall that if $a,b,c \in \cO_K$ are coprime and satisfy the Fermat equation \eqref{fermateq} one can construct the Frey elliptic curve
	
	\begin{equation} \label{freycurve} E:=E_{a,b,c,p}: Y^2=X(X-a^p)(X+b^p). \end{equation}
	
	The $j$-invariant of this elliptic curve has the formula
	\begin{equation} \label{jinv} j(E) = \frac{2^4(b^{2p}-a^pc^p)}{(abc)^{2p}}.
	\end{equation}
	Knowledge that $j(E)$ is integral outside a finite set of primes implies that $abc$ is actually supported only on that set.
	
	In \cite{turcas2020}, the second author proved the following result.
	
	\begin{theorem}[\cite{turcas2020}*{Theorem 1.3}] \label{2divabc} Let $K$ be a quadratic imaginary number field of class number one. Assume Serre's modularity conjecture (see \cite{turcas2020}*{Conjecture 2.2}) holds over $K$. Then, for any prime
		$p \geq 19$, the equation
		$a^p+b^p+c^p=0$
		does not have solutions in coprime $a,b,c \in \cO_K \setminus \{0 \}$ such that $2 \mid abc$.
	\end{theorem}
	
	\noindent\textbf{Remark.} The only reason for which the restriction $2 \mid abc$ appears in the statement of the theorem is as follows. Let $E$ be the Frey curve associated to the putative solution to \eqref{fermateq} as in \eqref{freycurve}. The author of \cite{turcas2020} could only prove that $\rhobar_{E,p}$ is absolutely irreducible for $p \geq 19$ such that $2 \mid abc$. The absolute irreducibility is required in the hypothesis of Serre's modularity conjecture (see  \cite{turcas2020}*{Conjecture 2.2}) and, if one shows it, the rest of the proof of the theorem above goes through.
	
	\medskip
	
	\noindent\textbf{Remark.} Previous results on the Asymptotic Fermat's Last Theorem over general fields that are not necessary totally real (see \cite{ASiksek}*{Theorems 1.1 and 1.2} or the theorems in \cite{FrSiKr}) are all stated for number fields $K$ which contain a prime $\mathfrak q$ with residue field $\F_2$ above $2$. One of the reasons for this restriction is that, over general $K$, the authors of the previously mentioned works had to assume that $E$ has a fixed prime of potentially multiplicative reduction $\mathfrak q$ in order to prove that $\rhobar_{E,p}$ is irreducible for $p$ larger than a constant $C_{K,\mathfrak q}$. If $\mathfrak q$ has residue field $\F_2$, considering $a^p+b^p+c^p=0 \pmod{\mathfrak q}$, one sees that $\mathfrak q \mid abc$. Using the formula \eqref{jinv} it can be easily deduced that $v_{\mathfrak q}(j(E)) <0$ for $p$ large enough and $\mathfrak q$ is the desired prime of multiplicative reduction for $E$.
	\medskip

	As an application to Theorem \ref{mainquad}, we prove the following version of the Asymptotic Fermat's Last Theorem for at least half of the prime exponents $p$.
	
	\begin{theorem} \label{thm:application}
		Let $K$ be a quadratic imaginary field of class number one and suppose that Serre's Modularity Conjecture (see Conjecture 2.2 in \cite{turcas2020}) holds for $K$. Fix $S$, a finite set of rational primes containing $2$, $3$ and $5$. There exists a constant $C_{S}$, depending only on $S$, such that for all $p > C_{S}$, if the following hold:
		\begin{enumerate}
			\item  $p \equiv 1 \pmod{3}$ or $p$ is a prime that splits in $K$ and such that $p \equiv 3 \pmod{4}$;
			
			\item  $a,b,c \in \cO_K \setminus \{0 \}$ are coprime such that $a^p+b^p+c^p =0$;
			
			\item If $l \notin S$ is a rational prime that divides $\norm(abc)$, then $l$ is inert in $K$. 
		\end{enumerate}
		then $K = \Q(\sqrt{-3})$ and the triple $(a,b,c)$ is, up to reordering, $(1,\epsilon, \epsilon^2)$, where $\epsilon$ is a non-trivial third root of unity.
	\end{theorem}
	We note that the larger $S$ is taken, the less restrictive hypothesis (3) becomes. However, this comes at the cost of (probably) increasing the constant $C_{S}$.
	
	This result is aligned with a more general Fermat Conjecture. We refer the interested reader to the discussion on page 2 of \cite{FrSiKr}, where the authors point out that the Fermat Conjecture is a consequence of the $abc$-conjecture for number fields.
	
	\medskip

\noindent\textbf{Remark.} We note that the result in the theorem above holds over general number fields $K$ if one includes the additional assumption that there are no weight $2$ complex eigenforms for $\GL_2$  of level $\mathcal N_K$, the latter quantity depending only on the number field $K$. We refer to Section 2 in \cite{ASiksek} for a precise definition of these eigenforms and to Lemma 5.3 in the same article for more details regarding the level $\mathcal N_K$.

	\section{Formal immersions and the proof of Theorems \ref{mainquad} and \ref{maind}}
	\label{sec:fomral}
	
	In this section we prove Theorems \ref{mainquad} and \ref{maind}.
	In particular, we prove that if an elliptic curve has multiplicative reduction at the prime above some totally inert rational prime $q$, then it cannot have an isogeny of large prime degree $p$.

	The idea is essentially to use the method going back to Mazur and Kamienny (see for example \cite{Kamienny-IMRN}), as modified by Merel \cite{MerelB} (see also \cite{Parent}). All the aforementioned papers use the fact that if the elliptic curve $E$ has a point of large order $n$ then it forces the curve to have a prime $\fq$ such that the curve has multiplicative reduction over it and all of its Galois conjugates. Then it follows that the reduction mod $\fq$ of a putative non-cuspidal point on $X_0(p)$ corresponding to this curve will be the same as the reduction mod $\fq$ of a cusp (the same is true for all Galois conjugates of $\fq$). This is then shown to be impossible by proving that a certain map from a symmetric power into a quotient (the Eisenstein quotient in Mazur's and Kamienny's papers and the \textit{winding quotient} in Merel's) of the Jacobian $J_0(n)$ of $X_0(n)$, which is of rank 0 over $\Q$, is a formal immersion at the aforementioned cusp modulo $\fq$.
	
	This method cannot be applied to elliptic curves with isogenies without assuming anything else about the curve. The existence of an isogeny of arbitrarily large degree does not force multiplicative reduction at any prime over number fields, as is easily seen on elliptic curves with complex multiplication. However, if one supposes multiplicative reduction at the prime $\fq$ above the totally inert rational prime $q$ that is not very small (as compared to the degree of the number field) and the existence of an isogeny of large prime degree $p$, 
then one can arrive at a contradiction using basically the same argument as before.
	
	Let $X_0(p)$ be the classical modular curve of level $p$, whose non-cuspidal $K$-rational points represent isomorphism classes of pairs $(E,C)$ of elliptic curves $E$ together with a $K$-rational subgroup $C$ of order $p$. The curve $X_0(p)$ has two cusps. We follow the convention (as in \cite{DR}) that $\infty$ is the cusp which is unramified under the $j$-map $X_0(p)\rightarrow X_0(1)$. This map is ramified of degree $p$ at the cusp $0$. With this convention, $0$ represents a a N\'eron $p$-gon and $\infty$ represents  and a N\'eron $1$-gon.

\begin{assumption}

Fix the following notation and assumptions for this section. Let $\fq$ be the (unique) prime above the totally inert rational prime $q$ and let $E/K$ have multiplicative reduction at $\fq$ and a $K$-rational subgroup $C$ of order $p$. 
Let $\sigma_1, \ldots, \sigma_d$ be the embeddings of $K$ into $\bar K$. Let $x\in X_0(p)(K)$ be the point corresponding to $(E,C)$, and let $y=(x^{\sigma_1}, \ldots, x^{\sigma_d})\in X_0^{(d)}(p)(\Q)$ be the point on the symmetric $d$-th power of $X_0(p)$.

\end{assumption}
	
	\begin{proposition} \label{prop-red}
		 The point $y\in X_0^{(d)}(p)(\Q)$ reduces to $(\infty,\ldots, \infty)_{\F_q}$ after possibly applying an appropriate Atkin-Lehner involution.
	\end{proposition}
	\begin{proof}
Let 
$x\in X_0(p)(K)$ be a point representing the pair $(E,C)$, where $C$ is a $K$-rational subgroup of order $p$ of $E$. Denote by $\widetilde{E}$ and $\widetilde{C}$ the reductions of $E$ and $C$ modulo $\fq$, respectively.

Since $q$ is totally inert in $K$, $y$ obviously reduces either to $(\infty,\ldots, \infty)_{\F_q}$, in which case we are done, or $(0,\ldots, 0)_{\F_q}$ after which we apply an Atkin-Lehner involution to obtain the desired result.




	\end{proof}
	
	Let $J_0^e(p)$ be the winding quotient of $J_0(p)$, as defined in \cite{MerelB}.
	
	\begin{theorem}[{Merel, \cite[Proposition 1]{MerelB}}]
		The rank of the winding quotient $J_0^e(p)(\Q)$ is 0.
	\end{theorem}
	
	
	Define $f_d:X_0^d(p)\rightarrow J_0^e$ to be the composition of the natural map $$X_0^d(p)\rightarrow J_0(p)$$
	$$ \quad (\alpha_1, \ldots, \alpha_d)\mapsto \left[\sum_{i=1}^{d}\alpha_i-d\infty\right]$$
	and the quotient map $J_0(p)\rightarrow J_0^e(p)$.
	
	\begin{theorem}[Parent]\label{propd}
		Suppose $d>2$, $q>d-1$ and $p> 65 (2d)^6$. Then the map $f_d:X_0^d(p)\rightarrow J_0^e(p)$ is a formal immersion at $(\infty, \ldots, \infty)_{\F_q}.$
	\end{theorem}
	\begin{proof}
		This follows from \cite[Theorem 4.18 and Section 5]{Parent} (see also \cite[Appendix A]{Derkamstoll}).
	\end{proof}

	For $d=2$ we can use the Mazur's Eisenstein quotient and Kamienny's results instead, as these are more explicitly stated. Using Eisenstein's quotient, we will see that the results we get are the best possible.
	
	Let $J_p$ be the Eisenstein quotient of $J_0(p)$ (see \cite{maz-orig} for the definition).
	
	\begin{theorem}[Mazur, {\cite[Theorem 4]{Mazur}}]
		The rank of $J_p(\Q)$ is zero.
	\end{theorem}
	
	Define now $f_d$ to be as defined before, with the difference that the quotient map maps to $J_p$ instead of $J_0^e(p)$.
	
	\begin{theorem}[{Kamienny, \cite[Proposition 3.2]{Kamienny1992}}] \label{prop2}
		Let $d=2$ and $q>5$, $p>71$. Then $f_d$ is a formal immersion at $(\infty,\infty)_{\F_\fq}$.
	\end{theorem}
	
	We now have all the ingredients needed to prove the theorems.
	
	\begin{proof}[Proof of Theorems \ref{mainquad} and \ref{maind}]
		To prove our claims, we use the following standard argument. Suppose that an elliptic curve $E$ with a $p$-isogeny (where $p$ satisfies the assumptions of the theorems) over a number field $K$ has multiplicative reduction at the unique prime $\fq$ of $K$ above the totally inert rational prime $q$. It corresponds to a non-cuspidal point $x\in X_0(p)(K)$. Let $y=(x^{\sigma_1}, \ldots, x^{\sigma_d})\in X_0^{(d)}(p)(\Q)$.
		
		The map $f_d$ is a formal immersion at $(\infty,\ldots, \infty)_{\F_q}$ by Theorems \ref{propd} and \ref{prop2}. It follows that there is at most one element in $X_0(p)^{(d)}(\Q)$ that reduces modulo $q$ to $(\infty,\ldots, \infty)_{\F_q}$ (see for example \cite[Lemma 3.1]{Derkamstoll}). Since we know that $\left(\infty,\ldots, \infty\right)\in X_0^{(d)}(p)(\Q)$ reduces to $(\infty,\ldots, \infty)_{\F_q}$ and, by Proposition \ref{prop-red}, we have that $y$ reduces to $(\infty,\ldots, \infty)_{\F_q}$ (after possibly applying an Atkin-Lehner involution), we conclude that $y=\left(\infty,\ldots, \infty\right)$. This contradicts the hypothesis that $x \in X_0(p)(K)$ is non-cuspidal.
	\end{proof}

	\section{Mod $p$ Galois representations and the proof of Theorem \ref{thm:application}}
	
	Before giving the proof for our Diophantine result, let us bring to the reader's attention the following theorem of Larson and Vaintrob.
	
	\begin{theorem}[\cite{LarVain13Jus}*{Theorem 1}] \label{larvain} Let $K$ be a number field. There exists a finite set of primes $M_K$, depending only on $K$, such that for any prime $p \notin M_K$ and any elliptic curve $E/K$ for which $\rhobar_{E,p} \otimes \overline{\F}_p \sim \left( \begin{array}{cc} \lambda & * \\ 0 & \lambda' \end{array} \right)$ where $\lambda, \lambda' : G_K  \to \overline{\F}_p^{\times}$ are characters, one of the following happens.
		\begin{enumerate}		
			\item There exists a CM elliptic curve $E'/K$ , whose CM field is contained in $K$, with $\rhobar_{E',p}\otimes \overline{\F}_p \sim \left( \begin{array}{cc} \theta& 0 \\ 0 & \theta' \end{array} \right) $ and such that $\lambda^{12} = \theta^{12}$.
			\item The Generalized Riemann Hypothesis fails for $K(\sqrt{-p})$, and $\lambda^{12}= \chi_p^6$, where $\chi_p: G_K \to \F_p^{\times}$ is the mod $p$ cyclotomic character. Moreover, in this case $\rhobar_{E,p}$ is already reducible over $\F_p$ and $p \equiv 3 \pmod{4}$.
		\end{enumerate}
	\end{theorem}
	
We can now start proving Theorem \ref{thm:application}. The quadratic imaginary fields of class number one are $K=\mathbb Q(\sqrt{-d})$, where $d$ is one of 1, 2, 3, 7, 11, 19, 43, 67 or 163. The cases $d \in \{1,2,7 \}$ follow by the more general results in \cite{turcas2018}. From now on we will assume that $K$ is one of the six remaining fields in which the prime $2$ is inert.
	
	Suppose that a triple $a,b,c \in \cO_K \setminus \{0 \}$ which satisfies the hypothesis (2) and (3) of Theorem \ref{thm:application} exists and assume that $p \geq 19$.  From Theorem \ref{2divabc}, we know the following:
	\begin{itemize}
		\item The prime ideal $\mathfrak q = 2\cO_K$ does not divide $abc$.
		\item If $E: Y^2=X(X-a^p)(X+b^p)$ is the Frey elliptic curve
		attached to this solution,  then $E$ has additive potentially good reduction at $\mathfrak q := 2 \cO_K$ and $E$ is semistable at every prime ideal $\mathfrak q' \neq \mathfrak q$.
		\item From the proof of the aforementioned theorem, it follows that the mod $p$ Galois representation $\rhobar_{E,p}$ is absolutely reducible and unramified outside the primes above $2p$.
	\end{itemize}
	
	Let us assume that $C_{S} \geq 53$. In this case, it follows from \cite{Serre81}*{Lemme 18'} that the image of $\rhobar_{E,p}(G_K)$ under the projection $\GL_2(\F_p) \to \PGL_2(\F_p)$ contains an element whose order is at least 13. We remark that, although Lemme 18' in loc. cit. is stated for $K=\Q$, its proof remains valid in the setting of general number fields $K$, as long as $p$ does not ramify in $K$. This guarantees that the image of $\rhobar_{E,p}(G_K)$ in $\PGL_2(\F_p)$ does not lie in any of its exceptional subgroups $A_4$, $S_4$ or $A_5$ (see \cite{serre72}*{Subsection 2.5}).
	
	From the discussion in \cite{serre72}*{Section 2} we deduce that if $\rhobar_{E,p} \otimes_{\F_p} \overline{\F}_p$ is diagonalisable, then the image $\rhobar_{E,p}(G_K)$ is contained in a Cartan subgroup of $\GL_2(\F_p)$, distinguishing the following two cases.
	
	Suppose first that $\rhobar_{E,p}$ is irreducible, but absolutely reducible. In this situation, it follows that the image $\rhobar_{E,p}(G_K)$ is contained in a Cartan non-split subgroup. We have that
	$\rhobar_{E,p} \otimes_{\F_p} \overline{\F}_p \sim \left( \begin{array}{cc} \lambda & 0 \\ 0 & \lambda^p \end{array} \right),$
	where $\lambda : G_K \to \mathbb F_{p^2}^{\times}$ is a character. The latter is not $\mathbb F_p$-valued and $\lambda^{p+1} = \chi_p$, where $\chi_p : G_K \to \F_p^{\times}$ is the mod $p$ cyclotomic character.
	
	Let us write $I_{\mathfrak q}$ for the inertia subgroup of $G_K$ corresponding to $\mathfrak q$. From \cite{AFreitasSiksek}*{Lemma 3.7} we know that 3 divides the order of $\rhobar_{E,p}(I_{\mathfrak q})$.  There exists an element $g \in I_{\mathfrak q}$ such that $\rhobar_{E,p}(g)$ has order exactly $3$ by Cauchy's theorem. This implies that $\lambda(g)$ has order $3$ in $\F_{p^2}^{\times}$. On the other hand, it is known that $\chi_p$ is unramified outside the places above $p$ and since $g \in I_{\mathfrak q}$, we have $\lambda^{p+1}(g)=1$. The latter implies that $3 \mid p+1$ and this, combined with the hypothesis of our theorem, gives that $p \equiv 3 \pmod{4}$ and that it must split in $K$.
	
	Let us now apply Theorem \ref{larvain} to our Frey curve $E$.
	To make sure the hypothesis of the aforementioned theorem is satisfied, we will assume that $C_{S}$ and hence $p$ is larger than any of the primes in $M_K$, for all fields $K$ under discussion. As we assume that $\rhobar_{E,p}$ is not reducible but only absolutely reducible, the first case of Theorem \ref{larvain} applies.
	
	There is an elliptic curve $E'/K$ with CM by an order in $K$ such that $$\rhobar_{E',p} \otimes \overline{\F}_p \sim \left( \begin{array}{cc} \theta & 0 \\ 0 & \theta' \end{array} \right) \text{ and }\lambda^{12} = \theta^{12}.$$ We recall that as $p$ is supposed to be split in $K$ from the theory of elliptic curves with CM, we know that the image $\rhobar_{E',p}(G_K)$ is contained inside a split Cartan subgroup. The character $\theta: G_K \to \overline{\mathbb F}_p^{\times}$ is in fact $\mathbb F_p$-valued. 	
	
	If needed, we now increase $C_{S}$ such that $C_{S}> 163$ to ensure that $p$ is unramified in $K$. Let $\mathfrak p$ be any of the two primes lying above $p$. From \cite{LarVain13Jus}*{Remark 1.1} we know that the character $\lambda \theta^{-1}$ is unramified outside the primes of additive reduction of $E$, in particular $\lambda \theta^{-1}$ is unramified at $\mathfrak p$. In other words, we have $\lambda\vert_{I_{\mathfrak p}} = \theta\vert_{I_{\mathfrak p}}$, where $I_{\mathfrak p}$ is the inertia subgroup of $\mathfrak p$. Recall that $\lambda^{p+1} = \chi_p$, hence $\theta^{p+1}\vert_{I_{\mathfrak p}} = \chi_p\vert_{I_{\mathfrak p}}$.
	
	On one hand, $\theta$ is $\F_p$-valued, so its order divides $p-1$. We deduce that
	$$\chi_p^{(p-1)/2}\vert_{I_{\mathfrak p}} = \left(\theta^{p-1} \right)^{(p+1)/2} \vert_{I_{\mathfrak p}} =1.$$
	On the other hand, $p$ is unramified in $K$ so the character $\chi_p\vert_{I_{\mathfrak p}}$ surjects on $\mathbb F_p^{\times}$, therefore its order is $p-1$, giving a contradiction.

	We just showed that if $\rhobar_{E,p}$ is not absolutely irreducible, then it is reducible. Recall that $S$ is a finite set of rational primes containing $2$, $3$ and $5$. The hypothesis (3) of our theorem implies that if $l \notin S$ then $E$ has the same type of reduction, good or multiplicative, at all prime ideals $\mathfrak l$ above $l$. Moreover, (3) guarantees that if $\mathfrak l \mid l$ is a prime of multiplicative reduction for $E$, then $l$ is inert in $K$. 
From Theorem \ref{mainquad} it follows that the reduction is good at such primes $\mathfrak l$. This means that $j(E)$ is integral outside of $S$, in particular that $abc$ is supported only on primes lying above the ones in $S$.
	
	The Fermat equation can be written as $(-a/c)^p+ (-b/c)^p=1$. Observe that $(-a/c)^p$ and $(-b/c)^p$ are solutions to the $S$-unit equation
	\begin{equation} \label{Sunit}
	x+ y=1, \text{ where } x,y \in \cO_{K,S}^{\times}.
	\end{equation}
	
	Due to the famous Siegel's theorem, we know that \eqref{Sunit} has finitely many solutions. Suppose that one of $a$, $b$ or $c$ is not a unit. Without losing generality, we can assume that $a$ is divisible by some prime ideal $\mathfrak l$ of $K$. As \eqref{Sunit} has finitely many solutions, the possible valuations of $v_{\mathfrak l}(x)$ belong to a finite set. A contradiction can be reached by increasing the exponent $p$ and this shows that there is a bound $B_{K,S}$ such that if $p> B_{K,S}$ then $a,b,c$ are units in $\cO_K$.
	
	Let us now take $C_{S}$ to be larger than all the constants $B_{K,S}$. For each one of the number fields $K$, there are finitely many units in $\cO_K$. Checking all the possibilities we find that if $p$ and $a,b,c$ are as in the hypothesis of our theorem, the only possible solutions arise when $K= \mathbb Q(\sqrt{-3})$ and these are $(a,b,c)=(1,\epsilon,\epsilon^2)$, up to reordering.
	
	\medskip
	
	\textbf{Acknowledgements}
We are grateful to the anonymous referee for suggesting many improvements to the presentation and for pointing out a simplification to the proof of Theorem 1.5. We also thank Samuel Le Fourn, Nuno Freitas and Philippe Michaud-Rodgers for some clarifications and helpful comments.

	\bibliographystyle{amsplain}
	\bibliography{perf-pow}

\begin{bibdiv}
\begin{biblist}

\bib{DR}{inproceedings}{
      author={Deligne, P.},
      author={Rapoport, M.},
       title={Les sch\'emas de modules de courbes elliptiques},
        date={1973},
   booktitle={Modular functions of one variable, {I}{I} ({P}roc. {I}nternat.
  {S}ummer {S}chool, {U}niv. {A}ntwerp, {A}ntwerp, 1972)},
      series={Lecture Notes in Math.},
      volume={349},
   publisher={Springer},
     address={Berlin, Germany},
       pages={143\ndash 316},
}

\bib{Derkamstoll}{article}{
      author={{Derickx}, M.},
      author={{Kamienny}, S.},
      author={{Stein}, W.},
      author={{Stoll}, M.},
       title={{Torsion points on elliptic curves over number fields of small
  degree}},
        date={2017-07},
     journal={ArXiv e-prints},
      eprint={1707.00364},
}

\bib{FrSiKr}{article}{
      author={Freitas, N.},
      author={Kraus, A.},
      author={Siksek, S.},
       title={{Class field theory, Diophantine analysis and the asymptotic
  Fermat's Last Theorem}},
        date={2020},
        ISSN={0001-8708},
     journal={Adv. Math.},
      volume={363},
       pages={106964},
  url={http://www.sciencedirect.com/science/article/pii/S0001870819305791},
}

\bib{AFreitasSiksek}{article}{
      author={Freitas, N.},
      author={Siksek, S.},
       title={An asymptotic {F}ermat's {L}ast {T}heorem over five-sixths of
  real quadratic fields},
        date={2015},
     journal={Compos. Math.},
      volume={151},
       pages={1395\ndash 1415},
}

\bib{FreSikCriteria}{article}{
      author={Freitas, N.},
      author={Siksek, S.},
       title={Criteria for {I}rreducibility of mod $p$ {R}epresentations of
  {F}rey {C}urves},
        date={2015},
     journal={J. {T}h\' eor. {N}ombres {B}ordeaux},
      volume={27},
      number={1},
       pages={67\ndash 76},
}

\bib{Kamienny1992}{article}{
      author={Kamienny, S.},
       title={Torsion points on elliptic curves and q-coefficients of modular
  forms},
        date={1992Dec},
     journal={Invent. Math.},
      volume={109},
      number={1},
       pages={221\ndash 229},
         url={https://doi.org/10.1007/BF01232025},
}

\bib{Kamienny-IMRN}{article}{
      author={Kamienny, S.},
       title={Torsion points on elliptic curves over number fields of higher
  degree},
        date={1992},
     journal={Int. Math. Res. Notices IMRN},
       pages={129\ndash 133},
}

\bib{KarOz}{article}{
      author={Kara, Y.},
      author={Ozman, E.},
       title={Asymptotic {G}eneralized {F}ermat's {L}ast {T}heorem over
  {N}umber {F}ields},
        date={2020},
     journal={Int. J. Number Theory},
      volume={16},
      number={05},
       pages={907\ndash 924},
}

\bib{LarVain13Jus}{article}{
      author={Larson, Eric},
      author={Vaintrob, Dmitry},
       title={{Determinants of subquotients of Galois representations
  associated to abelian varieties}},
        date={2014},
     journal={J. Inst. Math. Jussieu},
      volume={13},
      number={3},
       pages={517\ndash 559},
}

\bib{Mazur}{article}{
      author={Mazur, B.},
       title={Rational points of abelian varieties with values in towers of
  number fields},
        date={1972},
     journal={Invent. Math.},
      volume={18},
       pages={183\ndash 266},
}

\bib{maz-orig}{article}{
      author={Mazur, B.},
       title={Modular curves and the {E}isenstein ideal},
        date={1978},
     journal={Inst. Hautes \'Etudes Sci. Publ. Math.},
       pages={33\ndash 186},
}

\bib{Mazur1978}{article}{
      author={Mazur, B.},
       title={Rational isogenies of prime degree},
        date={1978},
     journal={Invent. Math.},
      volume={44},
       pages={129\ndash 162},
}

\bib{MerelB}{article}{
      author={Merel, L.},
       title={Bornes pour la torsion des courbes elliptiques sur les corps de
  nombres},
        date={1996},
     journal={Invent. Math.},
      volume={124},
       pages={437–\ndash 449},
}

\bib{Parent}{article}{
      author={Parent, P.},
       title={Bornes effectives pour la torsion des courbes elliptiques sur les
  corps de nombres},
        date={1999},
     journal={J. Reine Angew. Math.},
      volume={506},
       pages={85\ndash 116},
}

\bib{ASiksek}{article}{
      author={{\c S}eng{\" u}n, M.~H.},
      author={Siksek, S.},
       title={On the asymptotic {F}ermat's {L}ast {T}heorem over number
  fields},
        date={2018},
     journal={Commentarii Matematici Helvetici},
      volume={93},
       pages={359\ndash 372},
}

\bib{serre72}{article}{
      author={Serre, J.~P.},
       title={Propri{\' e}t{\' e}s galoisiennes des points d'ordre fini des
  courbes elliptiques},
        date={1971/72},
     journal={Invent. Math.},
      volume={15},
       pages={259\ndash 331},
         url={http://eudml.org/doc/142133},
}

\bib{Serre81}{article}{
      author={Serre, J.~P.},
       title={Quelques applications du th\' eor\` eme de densit\' e de
  {C}hebotarev},
        date={1981},
     journal={Publ. Math. Inst. Hautes {\' E}tudes Sci.},
      volume={54},
       pages={323\ndash 401},
}

\bib{Siksek2007}{incollection}{
      author={Siksek, S.},
       title={The {M}odular {A}pproach to {D}iophantine {E}quations},
        date={2007},
   booktitle={Number {T}heory: {V}olume {II}: {A}nalytic and {M}odern {T}ools},
   publisher={Springer},
     address={New York, NY},
       pages={495\ndash 527},
}

\bib{turcas2018}{article}{
      author={{\c{T}}urca{\c{s}}, G.~C.},
       title={On {F}ermat's equation over some quadratic imaginary number
  fields},
        date={2018},
        ISSN={2363-9555},
     journal={Research in Number Theory},
      volume={4},
      number={2},
       pages={24 pages},
         url={https://doi.org/10.1007/s40993-018-0117-y},
}

\bib{turcas2020}{article}{
      author={{\c{T}}urca{\c{s}}, G.~C.},
       title={On {S}erre's modularity conjecture and {F}ermat's equation over
  quadratic imaginary fields of class number one},
        date={2020},
        ISSN={0022-314X},
     journal={J. Number Theory},
      volume={209},
       pages={516\ndash 530},
  url={http://www.sciencedirect.com/science/article/pii/S0022314X19302914},
}

\end{biblist}
\end{bibdiv}

\end{document}